\pdfoutput=1
\documentclass[11pt]{article}

\usepackage[utf8]{inputenc}
\usepackage[pdftex]{graphicx,xcolor}
\usepackage{array,fullpage,enumitem,microtype}
\usepackage[leqno]{amsmath}
\usepackage{amssymb,amsthm,mathtools,bm,stmaryrd,tikz-cd,mathdots}
\usepackage[mathscr]{euscript}
\usepackage[pdftex,colorlinks=true,allcolors=blue,unicode,bookmarksnumbered,psdextra]{hyperref}
\usepackage{bookmark,tocbibind,aliascnt}
\usepackage[capitalize,nosort]{cleveref}

\usepackage{mymacros}

\newcommand*\defn{\textbf}
\newcommand*\Lan{\mathrm{Lan}}
\newcommand*\Ran{\mathrm{Ran}}
\newcommand*\Colim{\mathrm{Colim}}
\newcommand*\op{{\!{op}}}
\newcommand*\Psh{\!{Psh}}
\newcommand*\Ind{\!{Ind}}
\newcommand*\Sind{\!{Sind}}
\newcommand*\Rec{\!{Rec}}
\newcommand*\yoneda{\!y}
\mathlig{<<}{\ll}
\mathlig{>>}{\gg}

\newlist{eqenum}{enumerate}{1}
\makeatletter
\let\c@eqenumi=\@undefined
\makeatother
\newaliascnt{eqenumi}{equation}
\setlist[eqenum]{label=(\theequation),align=left,labelindent=0pt,resume}
\crefname{eqenumi}{}{}

\begin{document}

\title{On sifted colimits in the presence of pullbacks}
\author{Ruiyuan Chen}
\date{}
\maketitle

\begin{abstract}
We show that in a category with pullbacks, arbitrary sifted colimits may be constructed as filtered colimits of reflexive coequalizers.
This implies that ``lex sifted colimits'', in the sense of Garner--Lack, decompose as Barr-exactness plus filtered colimits commuting with finite limits.
We also prove generalizations of these results for $\kappa$-small sifted and filtered colimits, and their interaction with $\lambda$-small limits in place of finite ones, generalizing Garner's characterization of algebraic exactness in the sense of Adámek--Lawvere--Rosický.
Along the way, we prove a general result on classes of colimits, showing that the $\kappa$-small restriction of a saturated class of colimits is still ``closed under iteration''.
\let\thefootnote=\relax
\end{abstract}

\section{Introduction}

A category is called \defn{sifted} if the category of cocones over any finite discrete family of objects in it is connected.
The significance of this notion is that sifted colimits are precisely those which commute with finite products in the category of sets.
Thus, sifted colimits exist in any finitary universal-algebraic variety and are computed on the level of the underlying sets.
For background on sifted colimits and their key role in categorical universal algebra, see \cite{ARsifted}, \cite{ARValgthy}.

The main examples of sifted colimits are filtered colimits and \defn{reflexive coequalizers}, i.e., coequalizers of parallel pairs of morphisms $X \rightrightarrows Y$ with a common section $Y -> X$.
It is well-known that these two types of colimits ``almost'' suffice to generate all sifted colimits.
To state this precisely, recall that by general principles \cite[5.35]{Kvcat}, every category $\!C$ has a \emph{free cocompletion} under any given class of colimits, which can be explicitly constructed as the full subcategory of the presheaf category $[\!C^\op, \!{Set}]$ on the closure of the representables under said colimits.
Let
\begin{align*}
\Sind(\!C) &:= \text{free cocompletion of $\!C$ under small sifted colimits}, \\
\Ind(\!C) &:= \text{free cocompletion of $\!C$ under small filtered colimits}, \\
\Rec(\!C) &:= \text{free cocompletion of $\!C$ under reflexive coequalizers}.
\end{align*}
Now Adámek--Rosický~\cite[2.3(2)]{ARsifted} (see also \cite[7.3]{ARValgthy}) showed that for $\!C$ with finite coproducts,
\begin{equation*}
\Sind(\!C) \simeq \Ind(\Rec(\!C));
\end{equation*}
the same equation was also shown for complete $\!C$ by Adámek--Rosický--Vitale~\cite[5.1]{ARValgex}.
It follows from this equation that if $\!C$ also has small sifted colimits, then those may be constructed as filtered colimits of reflexive coequalizers.
This then implies that a functor $F : \!C -> \!D$ preserving these latter types of colimits also preserves all sifted ones, as shown by Joyal~\cite[33.24]{Jlogos}, Lack~\cite[3.2]{LRlawvere}, and Adámek--Rosický--Vitale~\cite[2.1]{ARVsifted}.
However, some such assumption on $\!C$ as completeness or existence of finite coproducts is needed in all of these results: Adámek--Rosický--Vitale~\cite[\S1]{ARVsifted} give counterexamples for a general $\!C$.

The main results of this paper show that sifted colimits may be constructed as filtered colimits of reflexive coequalizers, in all the precise senses just described, assuming instead that $\!C$ has pullbacks.
In fact, we prove a ``relative'' version of this, where all colimits are bounded in size by some regular cardinal $\kappa \le \infty$.
The precise statements are given by \cref{thm:pb-sindk} and \cref{thm:pb-siftedk}.
These results are ultimately based on some interactions between pullbacks and sifted colimits of a purely combinatorial nature, that we consider in \cref{sec:rec,sec:sifted}.

In \cref{sec:lex}, we apply our main results to the richer setting where not only pullbacks but all $\lambda$-small limits exist (for some $\lambda \le \infty$ suitably related to $\kappa$), and these obey all compatibility or ``exactness'' conditions with the $\kappa$-small sifted colimits as hold in $\!{Set}$.
When $\lambda = \omega$, these conditions become ``lex sifted colimits'' in the sense of Garner--Lack~\cite{GLlex}; when $\lambda = \infty$, they become Adámek--Lawvere--Rosický's ``algebraically exact categories'' \cite{ALRalgex}.
By combining our main \cref{thm:pb-sindk} with known characterizations of various ``exactness'' conditions, we obtain that ``exactness'' between $\lambda$-small limits and $\kappa$-small sifted colimits may be reduced to four familiar conditions on quotients (i.e., Barr-exactness when $\lambda = \omega$) and filtered colimits; see \cref{thm:siftex}.
This generalizes Garner's~\cite{Galgex} characterization of algebraic exactness in the case $\lambda = \infty$.

In \cref{sec:colimk}, which is largely independent from the rest of the paper, we prove a general result on colimits, needed for our main results.
For a class of colimits $\Phi$ (e.g., the sifted ones), the aforementioned abstract construction of the free $\Phi$-cocompletion, as iterated $\Phi$-colimits of representable presheaves, can in certain cases be simplified by removing the need for iteration.
Such $\Phi$ are the \emph{saturated classes} of Albert--Kelly~\cite{AKsat}; sifted colimits were shown to form a saturated class in \cite{ARsifted}.
We show in \cref{thm:phik-lan} and \cref{thm:phik-lan-sat} that, roughly speaking, for any saturated class $\Phi$ and regular cardinal $\kappa$, the $\kappa$-small $\Phi$-colimits are still saturated.
This result, which boils down to a simple accessibility argument, plays a key role in the proof of the $\kappa$-small version of our main result, by providing an explicit description of the free $\kappa$-small sifted-cocompletion.
We also give one other application: we rederive, in \cref{thm:makkai-pare}, Makkai--Paré's \cite[2.3.11]{MPacc} ``retract-free'' characterization of $\lambda$-presentable objects in $\kappa$-accessible categories.

\paragraph*{Acknowledgments}

I would like to thank Richard Garner for pointing out the close connections of our work to \cite{ARValgex} and \cite{Galgex}, as well as the referee for several helpful suggestions which improved the presentation of \cref{sec:colimk}.
Research partially supported by NSF grant DMS-2054508.

\section{Reflexive coequalizers and pullbacks}
\label{sec:rec}

Throughout this paper, ``category'' will mean locally small category by default, so that we have a Yoneda embedding, denoted $\yoneda = \yoneda_\!C : \!C -> [\!C^\op, \!{Set}]$; we will sometimes treat $\yoneda$ as an inclusion.

We begin by describing the free reflexive-coequalizer cocompletion $\Rec(\!C)$ of a category with pullbacks $\!C$.
The construction is the same as that of Pitts (see \cite[\S2]{BCsymtop}, \cite[17.12]{ARValgthy}) for $\!C$ with finite coproducts.
Informally speaking, coproducts allow a coequalizer of coequalizers to be reduced to a single coequalizer, by taking the ``union'' of the edge-sets of the two graphs involved; when $\!C$ instead has pullbacks, the ``concatenation'' graph may be used instead of the ``union''.

By a \defn{graph} on an object $X$ in a category $\!C$, we will mean an arbitrary parallel pair $p, q : G \rightrightarrows X$ with codomain $X$; the graph is \defn{reflexive} if $p, q$ have a common section $r : X -> G$ (i.e., $pr = qr = 1_X$).
By abuse of terminology, we will often refer to the graph by $G$ instead of $p, q$.
For another graph $s, t : H \rightrightarrows X$, we say that $G$ is \defn{contained} in $H$ if $p, q$ jointly factor through $s, t$ via some morphism $f : G -> H$, i.e., $sf = p$ and $tf = q$.
If $\!C$ has pullbacks, the \defn{concatenation} of graphs $p, q : G \rightrightarrows X$ and $s, t : H \rightrightarrows X$ is the pullback
\begin{equation*}
\begin{tikzcd}
X &&[-1em] X &[-1em]& X \\[-1em]
& G \ular["p"] \urar["q"'] && H \ular["s"] \urar["t"'] \\[-1ex]
&& K = G \times_X H \ular["v"] \urar["w"']
\end{tikzcd}
\end{equation*}
regarded as a graph via $pv, tw : K \rightrightarrows X$.

We record the following easy facts about graphs, which we will freely use:

\begin{lemma}
\label{thm:graph}
\leavevmode
\begin{enumerate}[label=(\alph*)]
\item \label{thm:graph:refl}
A graph $p, q : G \rightrightarrows X$ is reflexive iff it contains the identity graph $1_X, 1_X : X \rightrightarrows X$.
\item \label{thm:graph:sub-coeq}
If a graph $G$ is contained in $H$, then a morphism coequalizing $H$ also coequalizes $G$.
\item \label{thm:graph:cat-coeq}
For graphs $G, H, K$ on $X$ fitting into a diagram as above (without $K$ necessarily being the pullback), any morphism coequalizing both $G, H$ also coequalizes $K$.
\item \label{thm:graph:refl-cat}
If $G, H$ are graphs on $X$, and $G$ is reflexive, then $H$ is contained in $G \times_X H$ and $H \times_X G$.
\item \label{thm:graph:refl-cat-coeq}
Thus, if $G, H$ are both reflexive, then so is $G \times_X H$, and for any functor $F : \!C -> \!D$, a morphism coequalizes $F(G \times_X H) \rightrightarrows F(X)$ iff it coequalizes both $F(G), F(H)$.
\end{enumerate}
\end{lemma}
\begin{proof}
\cref{thm:graph:refl-cat}
If $p, q$ have common section $r : X -> G$, then $s, t$ jointly factor through $pv, tw$ via $(rs, 1_H) : H -> G \times_X H$; similarly for $H \times_X G$.

\cref{thm:graph:refl-cat-coeq}
The first claim follows from \cref{thm:graph:refl,thm:graph:refl-cat}; the second follows from \cref{thm:graph:sub-coeq}, \cref{thm:graph:cat-coeq}, and \cref{thm:graph:refl-cat}.
\end{proof}

For an arbitrary category $\!C$, as noted in the Introduction, the free reflexive-coequalizer cocompletion $\Rec(\!C)$ may be constructed as the full subcategory of $[\!C^\op, \!{Set}]$ obtained by closing the representables under reflexive coequalizers, with the Yoneda embedding $\yoneda : \!C  -> \Rec(\!C) \subseteq [\!C^\op, \!{Set}]$ as unit.
In particular, $\Rec(\!C)$ contains the coequalizers, in $[\!C^\op, \!{Set}]$, of all reflexive graphs in $\!C$.

\begin{proposition}
\label{thm:pb-rec}
For a category with pullbacks $\!C$, the full subcategory of $[\!C^\op, \!{Set}]$ on the coequalizers of reflexive graphs in $\!C$ is already closed under reflexive coequalizers, hence is $\Rec(\!C)$.
\end{proposition}
\begin{proof}
Consider a reflexive parallel pair $f, g$, with common section $h$, between the coequalizers $U, V$ in $[\!C^\op, \!{Set}]$ of two reflexive graphs
$p, q : G \rightrightarrows X$ and $s, t : H \rightrightarrows Y$
in $\!C$:
\begin{equation}
\label{diag:pb-rec}
\begin{tikzcd}
\yoneda G \dar[shift right, "\yoneda p"'] \dar[shift left, "\yoneda q"] &[2em]
\yoneda H \dar[shift right, "\yoneda s"'] \dar[shift left, "\yoneda t"] \\
\yoneda X \dar[two heads] \rar[shift left=2, "\yoneda \~f"] \rar[shift right=2, "\yoneda \~g"'] &
\yoneda Y \dar[two heads] \lar[dashed, "\yoneda \~h"{anchor=center,fill=white,inner sep=0pt}] \\
U \rar[shift left=2, "f"] \rar[shift right=2, "g"'] &
V \lar[dashed, "h"{anchor=center,fill=white,inner sep=1pt}]
\end{tikzcd}
\end{equation}
As in \cite[\S2]{BCsymtop}, we may describe $f, g, h$ explicitly as follows: $f$ descends from a morphism $\yoneda X -> V$ (coequalizing $\yoneda p, \yoneda q$), which corresponds by the Yoneda lemma to an element of $V(X)$, i.e., an equivalence class of morphisms $\~f : X -> Y$ with respect to the equivalence relation generated by the graph $\!C(X, s), \!C(X, t) : \!C(X, H) \rightrightarrows \!C(X, Y)$.
Similarly, $g, h$ lift to some $\~g, \~h$ as shown.
To say that $fh = 1_V$ means that $\~f\~h : Y -> Y$ is equivalent to $1_Y$ via the equivalence relation generated by $\!C(Y, s), \!C(Y, t) : \!C(Y, H) \rightrightarrows \!C(Y, Y)$, which means they are connected by a ``homotopy in $H$'':
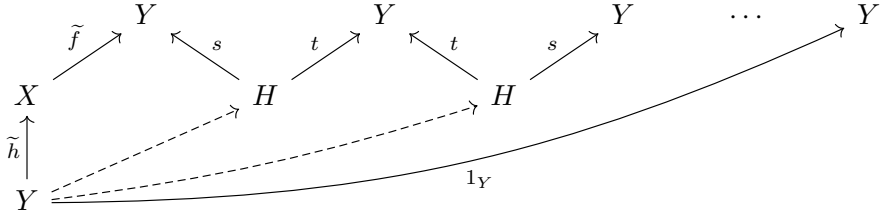
\begin{equation*}
\begin{tikzcd}
& Y && Y && Y & \dotsb & Y \\[-1ex]
X \urar["\~f"] && H \ular["s"'] \urar["t"] && H \ular["t"'] \urar["s"] \\[1ex]
Y \ar[u, "\~h"] \ar[urr, dashed] \ar[urrrr, bend right=5, dashed] \ar[uurrrrrrr, bend right=12, "1_Y"']
\end{tikzcd}
\end{equation*}
Similarly, $gh = 1_V$ means that $1_Y$ is connected via a ``homotopy'' to $\~g\~h$.
Pasting the latter ``homotopy'' to the left of the former one shows that the concatenation graph
\begin{equation*}
K := \dotsb \times_Y H \times_Y H^\op \times_Y X \times_Y H \times_Y H^\op \times_Y \dotsb \rightrightarrows Y
\end{equation*}
(where $H^\op$ is $H$ but with the roles of $s, t$ swapped) is reflexive.
Since $H$ is reflexive, $K$ contains the graph $\~f, \~g : X \rightrightarrows Y$.
Now concatenating $K$ once more with $H$ yields a reflexive graph $L$ on $Y$ which contains both $X$ and $H$ and is also a concatenation of copies of $X, H, H^\op$, hence has the same coequalizer in $[\!C^\op, \!{Set}]$ as the joint coequalizer of $\~f, \~g : X \rightrightarrows Y$ and $s, t : H \rightrightarrows Y$, which is easily seen to be the same as the coequalizer of $f, g : U \rightrightarrows V$ (see diagram \cref{diag:pb-rec} above).
\end{proof}

\section{Sifted categories with pullbacks}
\label{sec:sifted}

\begin{lemma}
\label{thm:sifted-pb-refl}
In a sifted category with pullbacks $\!C$, every graph $p, q : G \rightrightarrows X$ is contained in a reflexive graph $s, t : H \rightrightarrows X$.
\end{lemma}
\begin{proof}
Since $\!C$ is sifted, there is a zigzag connecting the cospans $G --->{p} X <---{1} X$ and $G --->{q} X <---{1} X$:
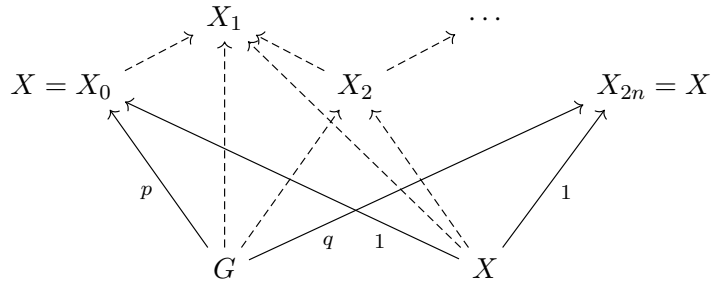
\begin{equation*}
\begin{tikzcd}
& X_1 && \dotsb \\[-1em]
\mathllap{X ={}} X_0 \urar[dashed] && X_2 \ular[dashed] \urar[dashed] && X_{2n} \mathrlap{{}= X} \\[3em]
& G \ar[ul, "p"] \ar[uu, dashed] \ar[ur, dashed] \ar[urrr, "q"'{pos=.2}] &&
X \ar[ulll, "1"{pos=.2}] \ar[uull, dashed] \ar[ul, dashed] \ar[ur, "1"']
\end{tikzcd}
\end{equation*}
Repeatedly replace each ``peak'' $X_i -> X_{i+1} <- X_{i+2}$ by its pullback, to get a single ``valley''
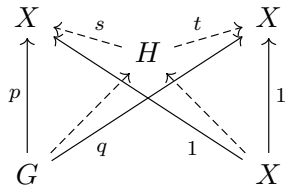
\begin{equation*}
\begin{tikzcd}[baseline=(\tikzcdmatrixname-3-1.base)]
X && X \\[-2em]
& H \ular[dashed, "s"'] \urar[dashed, "t"] \\[1em]
G \ar[uu, "p"] \ar[uurr, "q"'{pos=.2}] \ar[ur, dashed] &&
X \ar[uull, "1"{pos=.2}] \ar[uu, "1"'] \ar[ul, dashed]
\end{tikzcd}
\qedhere
\end{equation*}
\end{proof}

The following forms the combinatorial core of our main result (\cref{thm:pb-sindk}):

\begin{proposition}
\label{thm:sifted-pb-rec}
For a sifted category with pullbacks $\!C$, $\Rec(\!C)$ is filtered.
\end{proposition}
\begin{proof}
Clearly $\Rec(\!C)$ is nonempty because $\!C$ is.
Now let $U, V \in \Rec(\!C)$; by \cref{thm:pb-rec}, they are the coequalizers of (reflexive) graphs $p, q : G \rightrightarrows X$ and $s, t : H \rightrightarrows Y$ in $\!C$.
We may find a cospan over $U, V$ by finding a cospan $X -> Z <- Y$ in $\!C$, finding reflexive graphs on $Z$ containing the composite graphs $G \rightrightarrows X -> Z$ and $H \rightrightarrows Y -> Z$, concatenating them, and taking the reflexive coequalizer in $\Rec(\!C)$.
Given a parallel pair $f, g : U \rightrightarrows V$, as in the proof of \cref{thm:pb-rec}, we may lift them to $\~f, \~g : X \rightrightarrows Y$; we may find a morphism coequalizing $f, g$ by finding a reflexive graph $X'$ on $Y$ containing $\~f, \~g : X \rightrightarrows Y$, concatenating it with $H$, and taking the reflexive coequalizer in $\Rec(\!C)$, yielding the joint coequalizer of $X', H$ (see diagram \cref{diag:pb-rec}).
\end{proof}

\section{Saturated classes of $\kappa$-small colimits}
\label{sec:colimk}

As noted in the Introduction, the goal of this section is to prove a general result on colimits (\cref{thm:phik-lan} and \cref{thm:phik-lan-sat}), which roughly states that for any ``class of colimits $\Phi$ closed under iteration'', the $\kappa$-small $\Phi$-colimits are also ``closed under iteration''.
The study of such ``classes of colimits $\Phi$'' was initiated by Albert--Kelly~\cite{AKsat}, and our result here is a generalization of a result from that paper, \cite[7.4]{AKsat}, for $\kappa = \infty$.
We will use this to deduce (\cref{thm:sindk}) that $\kappa$-small sifted colimits are closed under iteration, from the known fact for $\kappa = \infty$ \cite[2.6]{ARsifted}; this is needed for our main result on sifted colimits.
The reader interested only in our main result may wish to take \cref{thm:sindk} as a black box and skip to the next section.

We begin by recalling the precise notion of a ``class of colimits $\Phi$''; see \cite{AKsat}, \cite{KScolim}.
However, our presentation differs slightly from these references, in that we do not initially restrict the weights in $\Phi$ to have small domain; we may therefore identify the saturation $\Phi^*$ with the free $\Phi$-cocompletion monad.
This is so that we may later discuss, in a uniform manner for all $\kappa \le \infty$, the case where $\Phi$ is generated by $\kappa$-small weights, thereby making clear the analogy between our results and \cite{AKsat}.
We will elaborate on this difference in viewpoint in \cref{rmk:phik-monad} below.

Recall \cite[\S3.4]{Kvcat} that given any category $\!J$ and presheaf $\phi \in [\!J^\op, \!{Set}]$, we may take the \defn{$\phi$-weighted colimit} $\phi \star F$ of a diagram $F : \!J -> \!C$, which is the same as the ordinary colimit of
\begin{equation*}
\yoneda_\!J \down \phi -> \!J --->{F} \!C,
\end{equation*}
i.e., $F$ applied to the canonical diagram over the category of elements $\yoneda_\!J \down \phi$ of $\phi$.
We will call $\phi$ a \defn{small presheaf} if $\!J$ is small; in that case, $\yoneda_\!J \down \phi$ is small, so the weighted colimit $\phi \star F$ is a small colimit.
More generally, we will call $\phi$ \defn{small-presented} if it is a small colimit of representables, in which case we can always take $\phi$ to be the $\phi|\!K^\op$-weighted colimit of the inclusion of a small full subcategory $\!K \subseteq \!J$; then a $\phi$-weighted colimit $\phi \star F$ is the same as the small colimit $\phi|\!K^\op \star F|\!K$.%
\footnote{Our ``small-presented'' is called ``accessible'' in \cite{Kvcat}, \cite{AKsat}, \cite{KScolim}, and ``small'' in many other works.
Our ``small'' is called such in \cite{Kvcat}, but is called a ``weight'' (as opposed to general presheaf) in \cite{KScolim} and other works.}

For any category $\!C$, let
\begin{equation*}
\Psh(\!C) \subseteq [\!C^\op, \!{Set}]
\end{equation*}
denote the full subcategory of small-presented presheaves, which is the free cocompletion of $\!C$ under all small colimits by \cite[5.35]{Kvcat}.
The universal property of $\Psh$ gives it the structure of a lax-idempotent 2-(pseudo)monad on the 2-category of all (locally small) categories, consisting of
\begin{eqenum}
\item \label{it:psh-unit}
for each category $\!C$, the unit $\yoneda_\!C : \!C -> \Psh(\!C)$;
\item \label{it:psh-mul}
for each $\!C$, the multiplication $\Lan_{\yoneda_{\Psh(\!C)}}(1_{\Psh(\!C)}) : \Psh(\Psh(\!C)) -> \Psh(\!C)$, taking $\phi |-> \phi \star 1_{\Psh(\!C)}$;
\item \label{it:psh-fun}
for $F : \!C -> \!D$, the induced cocontinuous functor $\Lan_{F^\op} : \Psh(\!C) -> \Psh(\!D)$, taking $\phi |-> \phi \star \yoneda F$;
\end{eqenum}
as usual for monads, given \cref{it:psh-unit}, we may combine \cref{it:psh-mul} and \cref{it:psh-fun} into
\begin{eqenum}
\item \label{it:psh-kleisli}
for $F : \!C -> \Psh(\!D)$, the Kleisli extension $\Lan_{\yoneda_\!C}(F) : \Psh(\!C) -> \Psh(\!D)$, taking $\phi |-> \phi \star F$.
\end{eqenum}
If $\yoneda_\!C$ has a partial left adjoint defined at some $\phi \in \Psh(\!C)$, the value must be the colimit $\phi \star 1_\!C$.
Thus the algebras of the monad, i.e., those $\!C$ for which $\yoneda_\!C$ has a total left adjoint, are precisely the cocomplete categories.
Similarly, the algebra homomorphisms are the cocontinuous functors.

Let $\Phi$ be a class of small-presented presheaves $\phi \in \Psh(\!C)$ on arbitrary categories $\!C$.
We identify such a $\Phi$ with the map taking each $\!C$ to the full subcategory
\begin{align*}
\Phi[\!C] := \Psh(\!C) \cap \Phi \subseteq \Psh(\!C).
\end{align*}
We conversely use $\Psh$ to name the class of \emph{all} small-presented presheaves (thus $\Psh(\!C)$ could also be denoted $\Psh[\!C]$, in accord with \cref{rmk:phik-monad} below).
A \defn{$\Phi$-colimit} means a colimit weighted by some $\phi \in \Phi$; a category $\!C$ is \defn{$\Phi$-cocomplete} if it has all $\Phi$-colimits; and a functor is \defn{$\Phi$-cocontinuous} if it preserves all $\Phi$-colimits.
The \defn{saturation} $\Phi^*$ of $\Phi$ is given by
\begin{equation*}
\Phi^*[\!C] := \text{closure of representables in $[\!C^\op, \!{Set}]$ under $\Phi$-colimits}.
\end{equation*}
Since every $\phi$ is the $\phi$-weighted colimit of representables $\phi \star \yoneda$, we have $\Phi \subseteq \Phi^*$.
By cocontinuity of $\star$ in the weight \cite[3.23]{Kvcat}, the class of weights $\psi$ for which a $\Phi$-cocomplete category is $\psi$-cocomplete, respectively, for which a $\Phi$-cocontinuous functor is $\psi$-cocontinuous, is closed under $\Phi$-colimits; thus
\begin{align*}
\text{$\Phi$-cocomplete} \iff \text{$\Phi^*$-cocomplete}, &&
\text{$\Phi$-cocontinuous} \iff \text{$\Phi^*$-cocontinuous}.
\end{align*}
In particular, $\Phi^*[\!C]$, being by definition closed under $\Phi$-colimits, is also closed under $\Phi^*$-colimits; that is, $\Phi^{**} = \Phi^*$, so that $\Phi |-> \Phi^*$ is a closure operation on the lattice of subclasses of $\Psh$.
Note that an equivalent definition of this closure operation is
\begin{equation*}
\Phi^* = \text{closure of $\Phi$ under the monad unit \cref{it:psh-unit} and Kleisli extension \cref{it:psh-kleisli} of $\Psh$}.
\end{equation*}
Thus the saturated classes $\Phi = \Phi^*$ are precisely the \defn{full submonads of $\Psh$} (borrowing terminology from \cite[\S3]{GLlex}), i.e., each $\Phi^*[\!C] \subseteq \Psh(\!C)$ is a full (replete) subcategory, and the monad operations of $\Psh$ restrict to $\Phi^*$, making it into a lax-idempotent 2-monad in its own right.
The $\Phi^*$-algebras are those $\!C$ for which $\yoneda : \!C -> \Phi^*[\!C]$ has a left adjoint, which means that $\phi \star 1_\!C$ exists for each $\phi \in \Phi^*[\!C]$; since $\Phi^*$ is closed under \cref{it:psh-fun}, this implies that
$\Lan_{F^\op}(\phi) \star 1_\!C = \phi \star F$ exists for each $F : \!J -> \!C$ and $\phi \in \Phi^*[\!J]$, i.e., that $\!C$ is $\Phi^*$-cocomplete (= $\Phi$-cocomplete).
Similarly, $\Phi^*$-algebra homomorphisms are precisely the $\Phi$-cocontinuous functors.

\begin{remark}
\label{rmk:phik-monad}
The free $\Phi$-cocompletion monad, which we denote $\Phi^*[-]$, is denoted $\Phi(-)$ in \cite{AKsat}, \cite{KScolim}, and many other works on classes of colimits.
These works always consider a class $\Phi$ to consist only of (what we are calling) \emph{small} presheaves; hence, they define the saturation $\Phi^*$ to consist of the presheaves in $\Phi(\!C)$ for \emph{small} $\!C$ only, which is an instance of what we will call ``$\kappa$-saturation'' below (see above \cref{thm:phik-lan-sat}), for $\kappa = \infty$.
With our more general classes of \emph{small-presented} presheaves, we have no need to distinguish between $\Phi^*[-]$ and $\Phi(-)$, and will hence never use the latter notation (except for standard named classes like $\Psh, \Sind$).
\end{remark}

Now let $\kappa \le \infty$ be an infinite regular cardinal (where $\infty$ is the bound on the fixed universe of small sets).
By \defn{$\kappa$-small}, we will generally mean of size $<\kappa$.
A \defn{$\kappa$-small presheaf} $\phi : \!C^\op -> \!{Set}$ is one where
(i) $\!C$ is $\kappa$-small, and
(ii) $\phi$ takes values in the full subcategory $\!{Set}_\kappa$ of $\kappa$-small sets;
note that these imply that $\yoneda \down \phi$ is $\kappa$-small, hence a $\phi$-weighted colimit is a $\kappa$-small colimit.

\begin{lemma}
\label{thm:phik-satk}
If $\Phi$ is a class of $\kappa$-small presheaves, and $\!C$ is $\kappa$-small, then $\Phi^*[\!C]$ consists of $\kappa$-small presheaves.
\end{lemma}
\begin{proof}
$[\!C^\op, \!{Set}_\kappa]$ contains the representables and is closed under $\kappa$-small, hence $\Phi$-, colimits.
\end{proof}

We now have the main result of this section, which extends \cite[7.4]{AKsat} to the case $\kappa < \infty$:

\begin{proposition}
\label{thm:phik-lan}
Let $\kappa \le \infty$ be uncountable regular, and let $\Phi$ be a class of $\kappa$-small presheaves.
Then for any $\!C$ and $\phi : \!C^\op -> \!{Set}$, we have $\phi \in \Phi^*[\!C]$ iff it is the left Kan extension of some $\psi \in \Phi^*[\!D]$ (which is $\kappa$-small, by the preceding lemma) for some $\kappa$-small subcategory $\!D \subseteq \!C$.
\end{proposition}

Informally speaking, the result holds because the free $\Phi$-cocompletion monad $\Phi^*[-]$ only adjoins operations of $\kappa$-small arity, hence ought to preserve $\kappa$-directed unions.
The first proof we give formalizes this idea, using 2-monad theory.
However, since 2-categorical machinery is not needed in the rest of the paper, we will also sketch for the reader's convenience a second, more direct proof.

\begin{proof}[Proof 1]
As is well-known, the pseudomonad $\Phi^*[-]$ is equivalent, as a pseudomonad, to a \emph{strict} 2-monad $\Phi_\!s$ which freely adjoins \emph{specified} $\Phi$-colimits to a category; see e.g., \cite{KLlim}.
Namely, pick a set $\@J$ of representatives of all $\kappa$-small categories $\!J$ up to isomorphism.
We may assume that $\Phi$ is already closed under transport along isomorphisms of categories.
We define $\Phi_\!s$ by declaring its (strict) algebras to be categories $\!C$ equipped with the following algebraic structure:
\begin{itemize}
\item  for each $\!J \in \@J$ and $\phi \in \Phi[\!J]$, a $\!J$-ary operation
$
\Colim_\phi : \!C^\!J -> \!C;
$
\item  for each $\!J \in \@J$, $\phi \in \Phi[\!J]$, $J \in \!J$, and $j \in \phi(J)$, a natural transformation
\begin{equation*}
\iota_{\phi,J,j} : \pi_J --> \Colim_\phi : \!C^\!J -> \!C
\end{equation*}
where $\pi_J$ is the $J$th projection,
obeying equational axioms saying that the family $(\iota_{\phi,J,j})_{J,j}$ forms (componentwise) a $\phi$-weighted cocone;
\item  for each $\!J \in \@J$ and $\phi \in \Phi[\!J]$, a natural transformation
\begin{equation*}
\gamma_\phi : \Colim_\phi \circ \pi_\!J --> \pi_\phi : \!C^{\!J[\phi]} -> \!C
\end{equation*}
where $\!J[\phi]$ denotes the union of $\!J$ and the single object $\phi$ in $[\!J^\op, \!{Set}]$ (so that a $\!J[\phi]$-indexed diagram is a $\!J$-indexed diagram together with a $\phi$-weighted cocone) and $\pi_\!J : \!C^{\!J[\phi]} -> \!C^\!J$ and $\pi_\phi : \!C^{\!J[\phi]} -> \!C$ are the projections,
obeying axioms saying that $\gamma_\phi$ is a cocone morphism from $(\iota_{\phi,J,j})_{J,j}$ to any other cocone, which is the identity when said other cocone is also $(\iota_{\phi,J,j})_{J,j}$.
\end{itemize}
Thus a $\Phi_\!s$-algebra is a category $\!C$ equipped with, for each $\!J \in \@J$, $\phi \in \Phi[\!J]$, and $F : \!J -> \!C$, a choice $\Colim_\phi(F)$ of $\phi$-weighted colimit $\phi \star F$ equipped with colimiting cocone $(\iota_{\phi,J,j,F})_{J,j}$; by our assumptions on $\Phi, \@J$, such a $\!C$ is then $\Phi$-cocomplete.

Now the strict 2-monad $\Phi_\!s$ which freely adjoins the above-defined structure to a category $\!C$ has rank $\kappa$, i.e., preserves (strict) $\kappa$-filtered colimits of categories, since the above structure is $\kappa$-ary.
Thus $\Phi_\!s(\!C)$ is the colimit of $\Phi_\!s(\!D)$ over all $\kappa$-small $\!D \subseteq \!C$.
We have a pseudonatural equivalence $\Phi_\!s \simeq \Phi^*$, given by arbitrarily fixing $\Phi$-colimits in each $\Phi^*[\!C]$.
Recall that a strict filtered colimit of categories is also a bicolimit (see e.g., \cite[Exposé~VI, \S6]{SGA4}), hence is preserved under equivalence.
Thus $\Phi^*[\!C]$ is the bicolimit of $\Phi^*[\!D]$ over all $\kappa$-small $\!D \subseteq \!C$, along the functors $\Phi^*[I_{\!D\!C}] = \Lan_{I_{\!D\!C}^\op}$ induced by the inclusions $I_{\!D\!C} : \!D -> \!C$, which implies the result.
\end{proof}

\begin{proof}[Proof 2]
$\Longleftarrow$ is because $\Phi^*$ is closed under \cref{it:psh-fun}.
For $\Longrightarrow$, since the conclusion is clearly satisfied by the representables, it suffices to check that if $\theta \in \Phi[\!J]$ and $F : \!J -> [\!C^\op, \!{Set}]$ such that each $F(J)$ is the left Kan extension of $\psi_J \in \Phi^*[\!D_J]$ for some $\kappa$-small $\!D_J \subseteq \!C$, then the conclusion also holds for $\theta \star F$.
Since $\!J$ is $\kappa$-small, the union of the $\!D_J$'s generates a $\kappa$-small $\!D \subseteq \!C$; by replacing each $\psi_J$ with its extension to $\!D$, we may assume each $\!D_J = \!D$ to begin with.
Let $I_{\!D\!C} : \!D -> \!C$ be the inclusion, and let $\Sigma$ be the $\kappa$-directed poset of all $\kappa$-small $\!D \subseteq \!E \subseteq \!C$.
For $J, K \in \!J$, we have
\begin{align*}
[\!C^\op, \!{Set}](F(J), F(K))
&= [\!C^\op, \!{Set}](\Lan_{I_{\!D\!C}^\op}(\psi_J), \Lan_{I_{\!D\!C}^\op}(\psi_K)) \\
&\cong [\!D^\op, \!{Set}](\psi_J, \Lan_{I_{\!D\!C}^\op}(\psi_K) \circ I_{\!D\!C}^\op) \\
\tag{$*$}
&\cong [\!D^\op, \!{Set}](\psi_J, \injlim_{\!E \in \Sigma} \Lan_{I_{\!D\!E}^\op}(\psi_K) \circ I_{\!D\!E}^\op) \\
\tag{$\dagger$}
&\cong \injlim_{\!E \in \Sigma} [\!D^\op, \!{Set}](\psi_J, \Lan_{I_{\!D\!E}^\op}(\psi_K) \circ I_{\!D\!E}^\op) \\
&\cong \injlim_{\!E \in \Sigma} [\!E^\op, \!{Set}](\Lan_{I_{\!D\!E}^\op}(\psi_J), \Lan_{I_{\!D\!E}^\op}(\psi_K)),
\end{align*}
where step ($\dagger$) is by $\kappa$-presentability of $\psi_J$, and step ($*$) is because for each $D \in \!D$ we have
\begin{align*}
\Lan_{I_{\!D\!C}^\op}(\psi_K)(D)
= \!C(D, I_{\!D\!C}-) \star \psi_K
\cong \injlim_{\!E \in \Sigma} \!E(D, I_{\!D\!E}-) \star \psi_K
= \injlim_{\!E \in \Sigma} \Lan_{I_{\!D\!E}^\op}(\psi_K)(D).
\end{align*}
Moreover, chasing through the above chain of isomorphisms, one sees that each colimit injection
\begin{align*}
[\!E^\op, \!{Set}](\Lan_{I_{\!D\!E}^\op}(\psi_J), \Lan_{I_{\!D\!E}^\op}(\psi_K))
--> [\!C^\op, \!{Set}](\Lan_{I_{\!D\!C}^\op}(\psi_J), \Lan_{I_{\!D\!C}^\op}(\psi_K))
\end{align*}
is given by $\Lan_{I_{\!E\!C}^\op}$, hence is compatible with the composition and identity operations between different hom-sets on either side.
Using this and $\kappa$-smallness of $\!J$, we may lift the entire diagram $F$ to a diagram $G : \!J -> [\!E^\op, \!{Set}]$ for some $\!E \in \Sigma$, taking values $G(J) = \Lan_{I_{\!D\!E}^\op}(\psi_J)$ on objects $J \in \!J$ (whence $G : \!J -> \Phi^*[\!E]$),
such that
$F \cong \Lan_{I_{\!E\!C}^\op} \circ G$.
Then $\theta \star F \cong \theta \star \Lan_{I_{\!E\!C}^\op}(G(-)) \cong \Lan_{I_{\!E\!C}^\op}(\theta \star G)$, with $\theta \star G \in \Phi^*[\!E]$.
\end{proof}

Let us also record the following generalization of \cref{thm:phik-lan} (by taking $\!J$ to be a singleton), needed for \cref{thm:sindk-lim} below, which follows from either of the two proofs above:

\begin{corollary}
\label{thm:phik-lan-diag}
Let $\kappa \le \infty$ be uncountable regular, and let $\Phi$ be a class of $\kappa$-small presheaves.
Then for any $\!C$ and $F : \!J -> \Phi^*[\!C]$ with $\!J$ $\kappa$-small, $F$ lifts to some $G : \!J -> \Phi^*[\!D]$ for some $\kappa$-small subcategory $\!D \subseteq \!C$ such that, letting $I : \!D -> \!C$ be the inclusion, we have $F \cong \Lan_{I^\op} \circ G$.
\qed
\end{corollary}

We now give an equivalent reformulation of \cref{thm:phik-lan}.
For any class $\Phi \subseteq \Psh$, let $\Phi_\kappa \subseteq \Phi$ consist of the $\kappa$-small presheaves in $\Phi$.
We have an adjunction $(-)^* \dashv (-)_\kappa$ between classes of $\kappa$-small presheaves and saturated classes of small-presented presheaves; call the induced closure operation $((-)^*)_\kappa$ \defn{$\kappa$-saturation}.
Explicitly, for a class of $\kappa$-small presheaves $\Phi$, by \cref{thm:phik-satk},
\begin{equation*}
(\Phi^*)_\kappa[\!C] = \begin{cases}
                       \Phi^*[\!C] &\text{if $\!C$ is $\kappa$-small}, \\
                       \emptyset &\text{otherwise}.
                       \end{cases}
\end{equation*}
Thus $\Phi$ is $\kappa$-saturated iff it consists of the $\kappa$-small presheaves $\Phi = \Psi_\kappa$ in some saturated class $\Psi$, iff (it consists only of $\kappa$-small presheaves and) for all $\kappa$-small $\!C$, $\Phi[\!C]$ contains the representables and is closed under $\Phi$-colimits.
(Taking $\kappa = \infty$ recovers the ``small'' notion of saturation used in \cite{AKsat}.)

\begin{corollary}
\label{thm:phik-lan-sat}
Let $\kappa \le \infty$ be uncountable regular, and let $\Phi = \Psi_\kappa$ be a $\kappa$-saturated class of $\kappa$-small presheaves, consisting of the $\kappa$-small presheaves in some saturated class $\Psi$.
Then for any $\!C$ and $\phi : \!C^\op -> \!{Set}$, the following are equivalent:
\begin{enumerate}[label=(\roman*)]
\item \label{thm:phik-lan-sat:iter}
$\phi \in \Phi^*[\!C] = (\Psi_\kappa)^*[\!C]$, i.e., $\phi$ is an iterated $\kappa$-small $\Psi$-colimit of representables;
\item \label{thm:phik-lan-sat:lan}
$\phi$ is the left Kan extension of some $\psi \in \Phi[\!D] = \Psi_\kappa[\!D]$ for some $\kappa$-small $\!D \subseteq \!C$, whence in particular, $\phi$ is a single $\kappa$-small $\Psi$-colimit of representables, namely of the diagram
\begin{equation*}
\yoneda_\!D \down \psi -> \!D \subseteq \!C --->{\yoneda_\!C} [\!C^\op, \!{Set}].
\end{equation*}
\end{enumerate}
\end{corollary}
\begin{proof}
This follows from \cref{thm:phik-lan}, since $\Phi^*[\!D] = (\Phi^*)_\kappa[\!D] = \Phi[\!D]$ by $\kappa$-saturation.
\end{proof}

\begin{example}
\label{ex:sindk}
Consider the sifted colimits.
For small $\!C$, it is known by \cite[2.6]{ARsifted} that the free sifted-cocompletion
$\Sind(\!C) \subseteq [\!C^\op, \!{Set}]$
consists of exactly those presheaves $\phi$ whose category of elements $\yoneda \down \phi$ is sifted.
Thus the saturated class $\Sind \subseteq \Psh$, with $\Sind[\!C] := \Sind(\!C)$, has $\Sind_\kappa$ consisting of those presheaves on a $\kappa$-small category with $\kappa$-small sifted category of elements, and so
\begin{equation*}
(\Sind_\kappa)^*[\!C] = \Sind_\kappa(\!C) := \text{free cocompletion of $\!C$ under $\kappa$-small sifted colimits}.
\end{equation*}
Applying \cref{thm:phik-lan-sat} with $\Psi := \Sind$ yields
\end{example}

\begin{corollary}
\label{thm:sindk}
For uncountable regular $\kappa \le \infty$ and any category $\!C$, the following are equivalent for a presheaf $\phi : \!C^\op -> \!{Set}$:
\begin{enumerate}[label=(\roman*)]
\item  $\phi \in \Sind_\kappa(\!C)$, i.e., $\phi$ is an iterated $\kappa$-small sifted colimit of representables;
\item  $\phi$ is the left Kan extension of some $\psi : \!D^\op -> \!{Set}_\kappa$ with sifted category of elements, for some $\kappa$-small $\!D \subseteq \!C$ (hence $\phi$ is a single $\kappa$-small sifted colimit of representables over $\yoneda_\!D \down \psi$).
\qed
\end{enumerate}
\end{corollary}

\begin{example}
We may of course similarly consider ($\kappa$-)filtered colimits.
For any infinite regular $\kappa < \lambda \le \infty$ and category $\!C$, let
\begin{equation*}
\kappa\Ind_\lambda(\!C) := \text{free cocompletion of $\!C$ under $\lambda$-small $\kappa$-filtered colimits}
\end{equation*}
(omitting $\lambda$ when $\lambda = \infty$ and $\kappa$ when $\kappa = \omega$).
For small $\!C$, it is well-known (see e.g., \cite[2.24]{ARlpac}) that $\kappa\Ind(\!C) \subseteq [\!C^\op, \!{Set}]$ consists of those $\phi$ with $\yoneda \down \phi$ $\kappa$-filtered.
Arguing exactly as above, we get
\end{example}

\begin{corollary}
\label{thm:indk}
For regular $\kappa < \lambda \le \infty$ and any $\!C$, the following are equivalent for $\phi : \!C^\op -> \!{Set}$:
\begin{enumerate}[label=(\roman*)]
\item  $\phi \in \kappa\Ind_\lambda(\!C)$, i.e., $\phi$ is an iterated $\lambda$-small $\kappa$-filtered colimit of representables;
\item  $\phi$ is the left Kan extension of some $\psi : \!D^\op -> \!{Set}_\lambda$ with $\kappa$-filtered category of elements, for some $\lambda$-small $\!D \subseteq \!C$ (hence $\phi$ is a $\lambda$-small $\kappa$-filtered colimit of representables over $\yoneda_\!D \down \psi$).
\qed
\end{enumerate}
\end{corollary}

\begin{remark}
The preceding result can also be proved using accessible category theory, in the special case where the ``sharply smaller'' relation $\kappa \lhd \lambda$ holds (see \cite[2.12]{ARlpac}, \cite[2.3]{MPacc}).
Indeed, every $\phi \in \kappa\Ind_\lambda(\!C) \subseteq \kappa\Ind(\!C)$, being an iterated $\lambda$-small colimit of representables, is $\lambda$-presentable in $\kappa\Ind(\!C)$, hence by \cite[2.3.11]{MPacc} a $\lambda$-small $\kappa$-filtered colimit of $\kappa$-presentable objects in $\kappa\Ind(\!C)$, i.e., of representables (at least assuming $\!C$ is Cauchy-complete).%
\footnote{The proof of \cite[2.3.11]{MPacc} does not depend on the assumption that $\!C$ is (essentially) small, i.e., that $\kappa\Ind(\!C)$ is $\kappa$-accessible instead of ``class-$\kappa$-accessible'' \cite{CRacc}.
Cauchy-completeness of $\!C$ is not needed either, since the proof of \cite[2.3.10]{MPacc} really shows that (in our notation) every $\lambda$-presentable $\phi \in \kappa\Ind(\!C)$ can be written as a retract of a $\lambda$-small $\kappa$-filtered colimit of (not just $\kappa$-presentables but) representables.}
Similar results on ``relatively accessible'' categories of the form $\kappa\Ind_\lambda(\!C)$ can be found in \cite{Lheart}.

However, the argument given above works for all $\kappa < \lambda$, and avoids the combinatorial proof of \cite[2.3.11]{MPacc}.
In fact, we can deduce \cite[2.3.11]{MPacc} from the above:
\end{remark}

\begin{corollary}[Makkai--Paré]
\label{thm:makkai-pare}
For regular $\kappa \lhd \lambda < \infty$ and any $\!C$, every $\lambda$-presentable $\phi \in \kappa\Ind(\!C)$ is a $\lambda$-small $\kappa$-filtered colimit of representables.
\end{corollary}
\begin{proof}
As in \cite[proof of 2.3.10 and following remark]{MPacc}, use $\kappa \lhd \lambda$ to write $\phi$ as a retract of a $\lambda$-small $\kappa$-filtered colimit of representables, which is in $\kappa\Ind_\lambda(\!C)$ since splitting an idempotent is a $\lambda$-small $\kappa$-filtered colimit, hence a $\lambda$-small $\kappa$-filtered colimit of representables by \cref{thm:indk}.
\end{proof}

\section{Sifted colimits in the presence of pullbacks}
\label{sec:main}

Recall from \cite[5.62]{Kvcat} (see also \cite[4.2]{KScolim}) the following characterization of categories which are the free cocompletion of a subcategory under some given class of small colimits $\Phi$.
(Here $\Phi$ could be a class of weights as in the previous section; but we will only need one case, where the colimits are of diagrams of certain shapes.)
For a category $\!C$ cocomplete under $\Phi$-colimits, an object $X \in \!C$ is \defn{$\Phi$-atomic} if $\!C(X, -) : \!C -> \!{Set}$ preserves $\Phi$-colimits.
Now such $\!C$ is equivalent, via the restricted Yoneda embedding
$\!C -> [\!D^\op, \!{Set}]$,
to the free $\Phi$-cocompletion of a full subcategory $\!D \subseteq \!C$ iff
\begin{enumerate}
\item[(i)]  every object in $\!D$ is $\Phi$-atomic in $\!C$, and
\item[(ii)]  the (replete) closure of $\!D$ under $\Phi$-colimits is all of $\!C$.
\end{enumerate}
This generalizes the standard characterization, for $\Phi =$ ``filtered colimits'', of finitely (class-)accessible categories $\!D = \Ind(\!C)$ as those generated under filtered colimits by finitely presentable objects.

We now have the main result of the paper.
As in the previous section, we write $\Sind_\kappa$ (resp., $\Ind_\kappa$) to denote free cocompletion under $\kappa$-small sifted (resp., filtered) colimits.

\begin{theorem}
\label{thm:pb-sindk}
For uncountable regular $\kappa \le \infty$ and a category with pullbacks $\!C$, we have
\begin{equation*}
\Sind_\kappa(\!C) \simeq \Ind_\kappa(\Rec(\!C))
\end{equation*}
via the restricted Yoneda embedding
$\Sind_\kappa(\!C) -> [\Rec(\!C)^\op, \!{Set}]$.
\end{theorem}
\begin{proof}
Since objects of $\Rec(\!C) \subseteq [\!C^\op, \!{Set}]$, being finite colimits of representables, are finitely presentable, hence also atomic with respect to $\kappa$-small filtered colimits (in $[\!C^\op, \!{Set}]$, hence also in the full subcategory $\Sind_\kappa(\!C)$ closed under $\kappa$-small filtered colimits), it suffices to show that every $\phi \in \Sind_\kappa(\!C)$ is a $\kappa$-small filtered colimit of objects in $\Rec(\!C)$.
By \cref{thm:sindk}, there is a $\kappa$-small subcategory $\!D \subseteq \!C$, which we may assume closed under pullbacks, and $\psi \in [\!D^\op, \!{Set}_\kappa]$ with sifted category of elements and left Kan extension $\phi$, so that $\phi$ is the $\kappa$-small sifted colimit of
\begin{equation*}
\yoneda_\!D \down \psi -> \!D \subseteq \!C --->{\yoneda_\!C} [\!C^\op, \!{Set}].
\end{equation*}
In other words, $\phi$ is the left Kan extension of this diagram along the unique functor $\yoneda_\!D \down \psi -> \*1$, hence is also the colimit of its reflexive-coequalizer-preserving left Kan extension to a diagram
\begin{equation*}
\Rec(\yoneda_\!D \down \psi) -> \Rec(\!C) \subseteq [\!C^\op, \!{Set}],
\end{equation*}
which is filtered by \cref{thm:sifted-pb-rec} since $\yoneda_\!D \down \psi$ inherits pullbacks from $\!D$, and is essentially $\kappa$-small since $\yoneda_\!D \down \psi$ is $\kappa$-small.
\end{proof}

\begin{corollary}
\label{thm:pb-siftedk}
For uncountable regular $\kappa \le \infty$, if a category with pullbacks $\!C$ has reflexive coequalizers and $\kappa$-small filtered colimits, then it has $\kappa$-small sifted colimits, and these are preserved by any functor $F : \!C -> \!D$ preserving reflexive coequalizers and $\kappa$-small filtered colimits.
\end{corollary}
\begin{proof}
Given a $\kappa$-small sifted diagram $G : \!J -> \!C$, a colimit of $G$ is the same thing as a colimit of $1_\!C$ weighted by
$\phi := \injlim (\yoneda_\!C G) \in \Sind_\kappa(\!C)$;
by writing $\phi$ as a $\kappa$-filtered colimit of reflexive coequalizers of representables (and using cocontinuity of weighted colimit $\star$ in the weight \cite[3.23]{Kvcat}), we get the colimit of $G$ as the corresponding $\kappa$-filtered colimit of reflexive coequalizers in $\!C$, which is hence preserved by any $F : \!C -> \!D$ preserving these latter colimits.
\end{proof}

\section{Lex sifted colimits and algebraic exactness}
\label{sec:lex}

The statement of \cref{thm:pb-siftedk} is somewhat peculiar in that the pullbacks required to exist in $\!C$ are not required to be compatible with the colimits in any way, nor are they required to be preserved by $F$.
Under such compatibility conditions, more can be said.
Garner--Lack~\cite{GLlex} have developed a general theory of what it means for a category to have finite limits and a given class of colimits $\Phi$ that obey all compatibility or ``exactness'' conditions between each other as hold in $\!{Set}$ (with an enriched generalization), known as \defn{$\Phi$-exactness}.
Such ``exactness'' conditions have also been considered for infinite limits and certain classes of colimits $\Phi$, including for sifted colimits by Adámek--Lawvere--Rosický~\cite{ALRalgex} under the name of \defn{algebraically exact} categories.
In this final section, we use our main results above to simplify, as well as extend with cardinality bounds, some of the known characterizations of ``exactness'' for sifted and related colimits.

A general form of ``exactness'' condition holding in $\!{Set}$ is that ``limits distribute over colimits''.
Distributivity of \emph{all} small limits over colimits in a complete and cocomplete category $\!C$ means that the weighted-colimit functor ${-} \star 1_\!C : \Psh(\!C) -> \!C$, which is the $\Psh$-algebra structure map on $\!C$, also preserves limits.
When considering limits and colimits of bounded size, the bounds need to be suitably chosen for the analogous distributivity condition to make sense.
For regular $\lambda, \kappa \le \infty$, following \cite[A.2.6.3]{Ltop}, we write
\begin{equation*}
\lambda << \kappa
\end{equation*}
to mean that for every $\lambda_0 < \lambda$ and $\kappa_0 < \kappa$, we have $\kappa_0^{\lambda_0} < \kappa$.
This ensures that a $\lambda$-small product of $\kappa$-small diagrams or presheaves is still $\kappa$-small, hence we may speak of $\lambda$-small limits distributing over $\kappa$-small colimits.
The following result says that we may likewise speak of $\lambda$-small limits distributing over $\kappa$-small \emph{sifted} colimits:

\begin{proposition}
\label{thm:sindk-lim}
For regular $\lambda << \kappa \le \infty$ with $\kappa$ uncountable, for a $\lambda$-complete category $\!C$, $\Sind_\kappa(\!C) \subseteq [\!C^\op, \!{Set}]$ is closed under $\lambda$-small limits; and for a $\lambda$-continuous $F : \!C -> \!D$ between two $\lambda$-complete categories, $\Lan_{F^\op} : \Sind_\kappa(\!C) -> \Sind_\kappa(\!D)$ is $\lambda$-continuous.
\end{proposition}
\begin{proof}
The second claim follows once we know the first, since for a $\lambda$-continuous $F : \!C -> \!D$, it is well-known that $\Lan_{F^\op} : \Psh(\!C) -> \Psh(\!D)$ is $\lambda$-continuous; see \cite[6.6]{DLlim}.

The case $\lambda = \kappa = \infty$ is \cite[3.11]{ALRalgex}.
More generally, if $\lambda = \kappa$, then since $\lambda << \kappa$, $\kappa$ is inaccessible, and so we may reduce to all of the cases $\lambda < \kappa$.

Consider the case $\lambda < \kappa = \infty$ with $\!C$ small.
If $\!C$ has finite coproducts, then $\Sind(\!C) \subseteq [\!C^\op, \!{Set}]$ consists of the finite-product-preserving functors \cite[2.8]{ARsifted}, which are closed under limits.
For a general small $\!C$, let $I : \!C -> \!D$ be a $\lambda$-continuous full embedding into a small $\lambda$-complete category with finite coproducts, e.g., the closure of the representables in $[\!C^\op, \!{Set}]$ under $\lambda$-small limits and finite coproducts.
Since $I$ is $\lambda$-continuous, as noted above, so is $\Lan_{I^\op}$.
Now for a $\lambda$-small diagram $F : \!J -> \Sind(\!C)$, we have $\Lan_{I^\op}(\projlim F) \cong \projlim(\Lan_{I^\op} \circ F) \in \Sind(\!D)$ since $\Lan_{I^\op} \circ F : \!J -> \Sind(\!D)$ and $\!D$ has finite coproducts.
This implies $\projlim F \in \Sind(\!C)$, by the following general fact:

\begin{lemma}
For any full and faithful $I : \!C -> \!D$ between small categories, $\phi \in [\!C^\op, \!{Set}]$ is in $\Sind(\!C)$ iff $\Lan_{I^\op}(\phi)$ is in $\Sind(\!D)$.
\end{lemma}
\begin{proof}
$\Longrightarrow$ is by pseudofunctoriality of $\Sind$.
For the converse, recall the following characterization of $\Sind$ from \cite[2.6]{ARsifted}: $\phi \in \Sind(\!C)$ iff $\Lan_{\yoneda_{\!C^\op}}(\phi) : [\!C, \!{Set}] -> \!{Set}$ preserves finite products.
So $\Lan_{I^\op}(\phi) \in \Sind(\!D)$ means that the left Kan extension of $\phi$ along the left-bottom composite
\begin{equation*}
\begin{tikzcd}
\!C^\op \dar["I^\op"'] \rar["\yoneda_{\!C^\op}"] &{} [\!C, \!{Set}] \dar["\Lan_I"] \\
\!D^\op \rar["\yoneda_{\!D^\op}"] &{} [\!D, \!{Set}]
\end{tikzcd}
\end{equation*}
preserves finite products.
Since the square commutes up to isomorphism, this is also the left Kan extension along the top-right composite.
Further left Kan extending along $(-) \circ I : [\!D, \!{Set}] -> [\!C, \!{Set}]$ recovers $\Lan_{\yoneda_{\!C^\op}}(\phi)$, since $I$ is full and faithful.
But since $(-) \circ I \dashv \Ran_I$, this further extension is the composite with the limit-preserving $\Ran_I$, whence $\Lan_{\yoneda_{\!C^\op}}(\phi)$ preserves finite products.
\end{proof}

Finally, consider the case $\lambda < \kappa$, $\lambda << \kappa \le \infty$, and $\!C$ possibly large.
Let $F : \!J -> \Sind_\kappa(\!C)$ be a $\lambda$-small diagram.
Since $\!J$ is $\kappa$-small, by \cref{thm:phik-lan-diag}, we may lift $F$ to a $G : \!J -> \Sind_\kappa(\!D)$ for some $\kappa$-small $\!D \subseteq \!C$, which we may assume to be closed under $\lambda$-small limits since $\lambda < \kappa$ and $\lambda << \kappa$, such that $F \cong \Lan_{I^\op} \circ G$ where $I : \!D -> \!C$ is the inclusion.
Since $\!D$ is small, we have $\projlim G \in \Sind(\!D)$ by the previous case, and also $\projlim G$ is a $\kappa$-small presheaf since each $G(J)$ was and $\lambda << \kappa$, whence $\projlim G \in \Sind_\kappa(\!D)$.
Again since $\Lan_{I^\op}$ is $\lambda$-continuous because $I$ is, we get $\projlim F \cong \projlim (\Lan_{I^\op} \circ G) \cong \Lan_{I^\op}(\projlim G) \in \Sind_\kappa(\!C)$.
\end{proof}

\begin{remark}
The preceding argument is not specific to $\Sind$, and can be applied to any class of colimits $\Phi = \Psi^+$ which is the class of ``$\Psi$-flat'' weights, in the sense of \cite{KScolim}, for a suitable ``sound'' class of limits $\Psi$ in the sense of \cite{ABLRdacc} or \cite[\S8]{KScolim}, yielding a generalization of \cite[6.3]{ABLRdacc} to $\lambda$-small limits.
As our focus is on $\Sind$, we leave the details to the interested reader.
\end{remark}

For regular $\lambda << \kappa \le \infty$, we define a category $\!C$ to be \defn{$(\lambda, \kappa)$-algebraically exact} if it
\begin{itemize}
\item  is $\lambda$-complete;
\item  has $\kappa$-small sifted colimits, or equivalently, $\yoneda_\!C : \!C -> \Sind_\kappa(\!C)$ has a left adjoint, which is then necessarily ${-} \star 1_\!C : \Sind_\kappa(\!C) -> \!C$; and
\item  ${-} \star 1_\!C$ is $\lambda$-continuous, i.e., ``$\lambda$-small limits distribute over $\kappa$-small sifted colimits''.
\end{itemize}
Informally, these conditions mean that $\!C$ is a ``quotient'', via ${-} \star 1_\!C$ which preserves $\lambda$-small limits and ($\kappa$-small sifted) colimits, of the ``subalgebra'' $\Sind_\kappa(\!C) \subseteq [\!C^\op, \!{Set}]$ of a ``power'' of $\!{Set}$, closed under those same operations.
It easily follows that all of the usual concrete ``exactness'' conditions relating specific types of limits and colimits are inherited from $\!{Set}$, such as those in the following characterization, which shows that a few familiar such conditions imply all others:

\begin{corollary}
\label{thm:siftex}
For regular $\lambda << \kappa \le \infty$ with $\kappa$ uncountable, a $\lambda$-complete category $\!C$ is $(\lambda, \kappa)$-algebraically exact iff it obeys the following four conditions:
\begin{enumerate}[label=(\roman*)]
\item \label{thm:siftex:ex}
Barr-exactness;
\item \label{thm:siftex:lex-filt}
$\kappa$-small filtered colimits exist and commute with finite limits;
\item \label{thm:siftex:prod-epi}
a $\lambda$-small product of regular epimorphisms is still a regular epimorphism;
\item \label{thm:siftex:prod-filt}
$\lambda$-small products distribute over $\kappa$-small filtered colimits, in the sense that for any $\lambda$-small family of $\kappa$-small filtered diagrams $(F_i : \!J_i -> \!C)_{i \in I}$, the canonical comparison morphism
\begin{equation*}
\injlim_{(J_i)_i \in {\prod_i} \!J_i} \prod_{i \in I} F(J_i) --> \prod_{i \in I} \injlim_{J \in \!J_i} F(J)
\end{equation*}
is an isomorphism.
\end{enumerate}
Moreover, a finitely continuous functor between two such categories preserves $\kappa$-small sifted colimits iff it is regular and preserves $\kappa$-small filtered colimits.
\end{corollary}

This result follows from combining the known characterizations of previous notions of ``exactness'' from \cite{GLlex}, \cite{ALRalgex}, and \cite{Galgex} with our \cref{thm:pb-sindk}.
We now discuss the precise relation between our definition of ``$(\lambda, \kappa)$-algebraic exactness'' and previous notions, and simultaneously indicate how to derive the various cases of \cref{thm:siftex} from what is known:
\begin{enumerate}[label=(\alph*)]

\item
When $\lambda = \omega$, we recover $\Phi$-exactness in the sense of \cite[3.4(3)]{GLlex} for $\Phi = \Sind_\kappa$, the class of (weights for) $\kappa$-small sifted colimits.

In this case, \cref{thm:siftex} follows from combining \cref{thm:pb-sindk} with the characterizations of $\Phi$-exactness from \cite{GLlex} for various $\Phi$.
Indeed, by the proof of \cite[5.10]{GLlex}, $\!C$ is $\Phi$-exact for $\Phi =$ the class of $\kappa$-small filtered colimits iff \cref{thm:siftex:lex-filt} above holds.
By \cite[5.12]{GLlex}, $\!C$ is exact for reflexive coequalizers iff it is \cref{thm:siftex:ex} Barr-exact, and also admits pullback-stable colimits of certain countable sequences $R_0 -> R_1 -> \dotsb$.
It is well-known and easily verified that this last condition is implied by \cref{thm:siftex:lex-filt}, and also that when $\lambda = \omega$, the conditions \cref{thm:siftex:ex}, \cref{thm:siftex:lex-filt} imply \cref{thm:siftex:prod-epi}, \cref{thm:siftex:prod-filt} respectively.
Thus, the conditions \cref{thm:siftex:ex}--\cref{thm:siftex:prod-filt} are together equivalent to exactness for the union of the classes of $\kappa$-small filtered colimits and reflexive coequalizers.
By \cref{thm:pb-sindk}, this union has saturation $\Sind_\kappa$, hence they determine the same exactness notion by \cite[3.4, 4.4]{GLlex}.

\item
When $\lambda = \kappa = \infty$, we recover algebraic exactness in the sense of \cite{ALRalgex}.
This case of \cref{thm:siftex} was conjectured by \cite{ALRalgex} and proved by Garner~\cite{Galgex}, and also follows from all of the cases $\lambda < \kappa = \infty$.

\item
Likewise, each case $\lambda < \kappa = \infty$ follows from all cases $\lambda << \kappa < \infty$, here using that there are arbitrarily large such $\kappa$ (see \cite[2.3.5]{MPacc}), and that $\Sind(\!C)$ is the union of $\Sind_\kappa(\!C)$, whence $(\lambda, \infty)$-algebraic exactness is the conjunction of $(\lambda, \kappa)$-algebraic exactness, over all such $\kappa$.

\item
For $\lambda << \kappa < \infty$, a notion of ``$\lambda$-algebraic exactness'' was also defined in \cite[2.1]{Galgex}, to mean that the Yoneda embedding into a certain full subcategory $\mathscr{S}_\kappa(\!C) \subseteq [\!C^\op, \!{Set}]$ has a $\lambda$-continuous left adjoint.
It is then proved in \cite[2.2]{Galgex} that \cref{thm:siftex} holds with this notion in place of our ``$(\lambda, \kappa)$-algebraic exactness''.

In fact, for uncountable $\kappa$, $\mathscr{S}_\kappa(\!C) \subseteq [\!C^\op, \!{Set}]$, which is defined to be the closure of the representables under $\lambda$-small limits, reflexive coequalizers, and $\kappa$-small filtered colimits, is the same as $\Sind_\kappa(\!C)$ by \cref{thm:pb-sindk} and \cref{thm:sindk-lim}.
So Garner's notion of algebraic exactness is the same as ours, and we get \cref{thm:siftex} as stated.

(Garner \cite[second paragraph of \S2]{Galgex} always takes $\kappa$ to be the \emph{least} cardinal $>> \lambda$, which is why his notations only involve a single cardinal parameter.
However, everything in \cite{Galgex} works just as well for $\lambda << \kappa$.
The case of least $\kappa >> \lambda$ was enough to derive the case $\lambda = \kappa = \infty$ of \cref{thm:siftex} as the main result of \cite{Galgex}, since \cref{thm:pb-sindk} was known for complete $\!C$ and $\kappa = \infty$ by \cite[5.1]{ARValgex}, as was \cref{thm:sindk-lim} for $\lambda = \kappa = \infty$ by \cite[3.11]{ALRalgex}.)

\item
In all cases, we may unravel the above arguments to arrive at the following concrete description of how a $\kappa$-small sifted colimit is computed using only the structure in \cref{thm:siftex:ex} and \cref{thm:siftex:lex-filt} above:
first, use \cref{thm:pb-sindk} to reduce to a $\kappa$-small filtered colimit of reflexive coequalizers;
then (as in the proof of \cite[5.12]{GLlex}) reduce each such reflexive coequalizer of a graph $G \rightrightarrows X$ to the quotient of $X$ by the equivalence relation generated by the image reflexive relation $R \subseteq X^2$ of $G$,
which is the colimit of the countable sequence $R \subseteq R \circ R^\op \circ R \subseteq \dotsb$ of composites of binary relations computed (as in any regular category) using pullback and image.

Each step of this procedure is preserved by a regular $F : \!C -> \!D$ preserving $\kappa$-small filtered colimits, which proves the last claim of \cref{thm:siftex}.

\end{enumerate}

\bigskip\noindent
Department of Mathematics and Statistics \\
CRM/McGill University \\
Montréal, QC, H3A 0B9, Canada \\
Email: \nolinkurl{ruiyuan.chen@umontreal.ca}


\begin{thebibliography}{00000000}

\bibitem[ABLR02]{ABLRdacc}  J.~Adámek, F.~Borceux, S.~Lack, and J.~Rosický, \emph{A classification of accessible categories}, J.\ Pure Appl.\ Algebra \textbf{175}~(2002), no.~1--3, 7--30.

\bibitem[ALR01]{ALRalgex}  J.~Adámek, F.~W.~Lawvere, and J.~Rosický, \emph{How algebraic is algebra?}, Theory Appl.\ Categ.\ \textbf{8}~(2001), no.~9, 253--283.

\bibitem[AR97]{ARlpac}  J.~Adámek and J.~Rosický, \emph{Locally presentable and accessible categories}, London Math Society Lecture Note Series, vol.~189, Cambridge University Press, 1997.

\bibitem[AR01]{ARsifted}  J.~Adámek and J.~Rosický, \emph{On sifted colimits and generalized varieties}, Theory Appl.\ Categ.\ \textbf{8}~(2001), no.~3, 33--53.

\bibitem[ARV01]{ARValgex}  J.~Adámek, J.~Rosický, and E.~M.~Vitale, \emph{On algebraically exact categories and essential localizations of varieties}, J.\ Algebra \textbf{244}~(2001), no.~2, 450--477.

\bibitem[ARV10]{ARVsifted}  J.~Adámek, J.~Rosický, and E.~M.~Vitale, \emph{What are sifted colimits?}, Theory Appl.\ Categ.\ \textbf{23}~(2010), no.~13, 251--260. 

\bibitem[ARV11]{ARValgthy}  J.~Adámek, J.~Rosický, and E.~M.~Vitale, \emph{Algebraic theories: A categorical introduction to general algebra}, Cambridge Tracts in Mathematics, vol.~184, Cambridge University Press, 2011.

\bibitem[AK88]{AKsat}  M.~H.~Albert and G.~M.~Kelly, \emph{The closure of a class of colimits}, J.\ Pure Appl.\ Algebra \textbf{51}~(1988), no.~1--2, 1--17.

\bibitem[AGV72]{SGA4}  M.~Artin, A.~Grothendieck, and J.~L.~Verdier, \emph{Théorie des topos et cohomologie étale des schémas}, Tome 2, Lecture Notes in Mathematics, vol.~270, Springer, 1971.

\bibitem[BC95]{BCsymtop}  M.~Bunge and A.~Carboni, \emph{The symmetric topos}, J.\ Pure Appl.\ Algebra \textbf{105}~(1995), no.~3, 233--249.

\bibitem[CR12]{CRacc}  B.~Chorny and J.~Rosický, \emph{Class-locally presentable and class-accessible categories}, J.\ Pure Appl.\ Algebra \textbf{216}~(2012), no.~10, 2113--2125.

\bibitem[DL07]{DLlim}  B.~J.~Day and S.~Lack, \emph{Limits of small functors}, J.\ Pure Appl.\ Algebra \textbf{210}~(2007), no.~3, 651--663.

\bibitem[Gar13]{Galgex}  R.~Garner, \emph{A characterisation of algebraic exactness}, J.\ Pure Appl.\ Algebra \textbf{217}~(2013), no.~8, 1421--1426.

\bibitem[GL12]{GLlex}  R.~Garner and S.~Lack, \emph{Lex colimits}, J.\ Pure Appl.\ Algebra \textbf{216}~(2012), no.~6, 1372--1396.

\bibitem[Joy08]{Jlogos}  A.~Joyal, \emph{Notes on logoi}, unpublished manuscript, 2008.

\bibitem[Kel82]{Kvcat}  G.~M.~Kelly, \emph{Basic concepts of enriched category theory}, London Math Society Lecture Note Series, vol.~64, Cambridge University Press, 1982.

\bibitem[KL00]{KLlim}  G.~M.~Kelly and S.~Lack, \emph{On the monadicity of categories with chosen colimits}, Theory Appl.\ Categ.\ \textbf{7}~(2000), no.~7, 148--170.

\bibitem[KS05]{KScolim}  G.~M.~Kelly and V.~Schmitt, \emph{Notes on enriched categories with colimits of some class}, Theory Appl.\ Categ.\ \textbf{14}~(2005), no.~17, 399--423.

\bibitem[LR11]{LRlawvere}  S.~Lack and J.~Rosický, \emph{Notions of Lawvere theory}, Appl.\ Categ.\ Structures \textbf{19}~(2011), no.~1, 363--391.

\bibitem[Low16]{Lheart}  Z.~L.~Low, \emph{The heart of a combinatorial model category}, Theory Appl.\ Categ.\ \textbf{31}~(2016), no.~2, 31--62.

\bibitem[Lur09]{Ltop}  J.~Lurie, \emph{Higher topos theory}, Annals of Mathematics Studies, vol.~170, Princeton University Press, 2009.

\bibitem[MP89]{MPacc}  M.~Makkai and R.~Paré, \emph{Accessible categories: the foundations of categorical model theory}, Contemporary Mathematics, vol.~104, American Mathematical Society, 1989.

\end{thebibliography}
\end{document}